\documentclass{article}

\usepackage[margin=1in]{geometry}
\usepackage{amssymb}

\usepackage{amsmath}

\usepackage{amsthm}
\usepackage{parskip}
\usepackage{setspace}
\begin{document}
\newtheorem{thm}{Theorem}[section]
\newtheorem{lem}[thm]{Lemma}
\newtheorem{cor}[thm]{Corollary}
\newtheorem{prop}[thm]{Proposition}
\newtheorem{defi}[thm]{Definition}

\title{Numerical Information of Hilbert Space Operators}

\author{Dylan M. Johnson \\ \small University of Iowa} 
\date{}

\maketitle

\begin{abstract}

For a Hilbert space operator $A\in B(H)$, the numerical range map of $A$, denoted $\phi_A$ is the map sending unit vectors $x$ in $H$ to $\langle Ax,x\rangle$. From this map we define the Numerical Information of $A$, denoted $\Sigma_A$, which is the sigma algebra induced on the unit sphere by the numerical range map. One of the main results of the paper are that if $A,B\in B(H)$, and $\phi_B=f\circ \phi_A$, then $f$ is affine. Our other main result is showing that $\Sigma_B\subset\Sigma_A$ if and only if $B=\alpha Re(A)+\beta Im(A)+\gamma I$ for some complex numbers $\alpha,\beta,\gamma$. 
\end{abstract}

47A12

\section{Introduction}
Let $H$ be a Hilbert space with inner product $\langle \cdot,\cdot\rangle$ and $B(H)$ denote the set of bounded linear operators on $H$. 
\begin{defi}\label{numrange}The numerical range of an operator $A$, denoted $W(A)$, is the set
\begin{equation*}
    W(A):=\{\langle Ax,x\rangle \;|\; ||x||=1\}.
\end{equation*}
\end{defi}
\begin{defi}\label{nummap}The numerical range map associated with $A$, denoted $\phi_A$ is defined as 
\begin{align*}
    \phi_A(x):S_H&\to \mathbb{C}\\
    x&\mapsto \langle Ax,x\rangle,
\end{align*}
where $S_H$ denotes the unit sphere in $H$.\end{defi}
The joint numerical range of a finite collection of operators $A_1,A_2,...,A_n\in B(H)$ is the set
\begin{equation*}
    W(A_1,...,A_n):=\{(\phi_{A_1}(x),...,\phi_{A_n}(x))\;|\;x\in S_H\}.
\end{equation*}
See \cite{joint} for a quality introduction to joint numerical ranges. \\\\
Another object of interest is the Numerical Shadow \cite{numshadow}, (or Numerical Measure \cite{nummeasure}). For a finite-dimensional Hilbert space, the Numerical Shadow of a matrix is the pushfoward of the Haar measure on $S_H$ induced by the numerical range map. For a measurable subset $U\in W(A)$ we define
\begin{equation*}
    \mu_A(U):=h(\phi_A^{-1}(U)),
\end{equation*}
where $h$ denotes the Haar measure on $S_H$. Joint numerical shadows and their applications were considered in \cite{jointshadow}.\\\\
Motivated by the numerical shadow construction, for each operator, we can associate the sigma algebra on the unit sphere induced by the numerical range map. Unlike the numerical shadow, the numerical information is defined for infinite-dimensional operators. Although we do not consider unbounded operators in this paper, the numerical information construction can be extended to include them. The numerical range of an unbounded operator is the whole complex plane, but the numerical information of two unbounded operators can be distinct.
\\\\Understanding these sigma-algebras should be helpful in solving problems in joint numerical ranges, numerical shadows, higher-rank numerical ranges \cite{compress}, inverse numerical range problems \cite{inverseproblem}, geometric problems, and quantum information problems \cite{quantum}.\\\\
One of the main results of the present work is giving necessary and sufficient conditions for the sigma algebra induced by one operator to be contained in the sigma algebra of another.\\\\
The layout of the paper is as follows. In Section \ref{tools} we provide the essential background on the numerical range map that we will use later in our proofs. Using the results of Sections \ref{preorder} and \ref{2dcase}, in Section \ref{conn} we prove that if $\phi_B=f\circ \phi_A$ for some $f:W(A)\to W(B)$, then $f$ is affine.  In Section \ref{numinfo}, we define the numerical information of an operator. This is the sigma algebra induced by the numerical range map. From our result on the connecting maps, we give necessary and sufficient conditions for the numerical information of one operator to be contained in the information of another. 
\subsection{Tools and Basic Results}\label{tools}
In this section, we present some basic results about numerical ranges and numerical range maps that will be used in later proofs. \\\\
The following is one of the most well-known results in the study of numerical ranges.\\\\
\textbf{Toeplitz-Hausdorff Theorem}\\
The numerical range of an operator is a convex set. \\
Proof: See \cite{yuangao}.
\begin{prop}\label{basics} For $A,B\in B(H)$ and $\lambda\in\mathbb{C}$\\
i) $\phi_{\lambda A}=\lambda\phi_A$\\
ii) $\phi_{A+B}=\phi_A+\phi_B$\\
iii) $\phi_{A}(x)=\lambda$ for all $x\in S_H$ if and only if $A=\lambda I$\\
iv) $W(A)\subset \mathbb{R}\iff A=A^*$\\
v) $W(U^*AU)=W(A)$ for all unitary $U$
\end{prop}
Proof: See \cite{yuangao}.\\\\
The numerical range map can be broken up into a real and imaginary part.
\begin{cor}\label{decomp} Using the above properties we have  \begin{equation*}    
\phi_A(x)=\phi_{Re(A)}(x)+i\phi_{Im(A)}(x)\end{equation*} where $Re(A):= \frac{A+A^*}{2}$ and $Im(A):=\frac{A-A^*}{2i}$
\end{cor}
\begin{lem}The numerical range map is invariant under scaling:\\
\begin{equation*}
    \phi_A(\lambda x)=\phi_A(x)\; \text{for}\; |\lambda|=1
\end{equation*}\end{lem}
Proof: See \cite{difftop}.\\\\
In later proofs, we will use the cardinality of the preimages of the numerical range map. When we say that the cardinality of a preimage is finite, we mean a finite number of linearly independent vectors. We acknowledge this using the phrase "up to scaling."\\\\
Compressing operators is an essential tool for many proofs involving the numerical range. 
\begin{defi}\label{compression} Compression of an operator. \end{defi}Let $M$ be a closed subspace of $H$, $P_M$ the orthogonal projection of $H$ onto $M$. For a linear operator $A$, the compression of $A$ to $M$ , denoted $A_M$, is the restriction to $M$ of the operator $P_MA$. 
\begin{lem}\label{complemma}For a closed subspace $M\subset H$, $W(A_M)\subset W(A)$.\end{lem}
Proof: We see that on $M$, $\phi_A(x)=\phi_{P_MA}(x)$. Take $x\in M$, 
    \begin{align*}
        \phi_{P_MA}(x)&=\langle P_MAx,x\rangle\\
        &=\langle Ax,P_M^*x\rangle\\
        &=\langle Ax,P_Mx\rangle\\
        &=\langle Ax,x\rangle\\
        &=\phi_A(x).
    \end{align*}
    From this we conclude that $W(A_M)\subset W(A).$ $\square$ 
\section{Preorder and Vector Space Induced By Connecting Maps}\label{preorder}
In this section, we begin our exploration of maps connecting numerical range maps. 
\begin{defi}\label{relation} We define a relation on $B(H)$, saying that $B\leq A$ if 
\begin{equation*}
    \phi_B=f\circ \phi_A
\end{equation*}where 
\begin{align*}
    f:W(A)&\to W(B)\\
    \langle Ax,x\rangle&\mapsto \langle Bx,x\rangle.
\end{align*}\end{defi}
We will often refer to $f$ as the connecting map from $A$ to $B$. In order for $f$ from above to be well defined we must have that 
\begin{equation*}
    \langle Ax,x\rangle=\langle Ay,y\rangle\text{ implies } \langle Bx,x\rangle=\langle By,y\rangle
\end{equation*}
for any $x,y\in S_H$. We use this fact several times throughout the paper. If we can find $x,y\in S_H$ such that $\phi_A(x)=\phi_A(y)$, but $\phi_B(x)\neq \phi_B(y)$, then $B\not\leq A$.\\\\
Note that a connecting map $f$ must be surjective as $\phi_B(x)=f\circ \phi_A(x)$ for all $x\in S_H$. \\\\
We will not use the following in this work, but it is an interesting fact. 
\begin{prop} The relation $\leq$ is a preorder. \end{prop}
Proof: For all $A\in B(H)$, 
\begin{equation*}
    \phi_A=Id\circ \phi_A,
\end{equation*}
where $Id$ denotes the identity map, $x\mapsto x$. Therefore $A\leq A$. \\\\
If $C\leq B$ and $B\leq A$, there exists functions $f,g:\mathbb{C}\to\mathbb{C}$ such that $\phi_B=f\circ \phi_A$ and $\phi_C=g\circ \phi_B$. As both $f$ and $g$ are surjective, we have
\begin{equation*}
    \phi_C=(g\circ f) \circ\phi_A,
\end{equation*}
and $C\leq A$. $\square$
\begin{defi}For all $A\in B(H)$, we consider the set 
\begin{equation*}
    V_A:=\{B\in B(H)\;|\; B\leq A\}.
\end{equation*}\end{defi}
\begin{lem}$V_A$ is an additive vector space. \end{lem}
Proof: We have shown that $A\leq A$, hence $V_A$ is non-empty.
\\\\
Given $B,C\in V_A,$ with respective connecting maps $f\;\&\; g$, $\lambda\in \mathbb{C}$, and $x\in S_H$, we have 
\begin{align*}
    \phi_{B+\lambda C}(x)&=\langle (B+\lambda C)x,x\rangle\\
    &=\langle Bx,x\rangle +\lambda\langle Cx,x\rangle\\
    &=\phi_B(x)+\lambda\phi_C(x)\\
    &=f\circ\phi_A(x)+\lambda g\circ\phi_A(x)\\
    &=(f+\lambda g)\circ\phi_A(x).
\end{align*}
Hence $B+\lambda C\leq A$. $\square$
\begin{lem}\label{velements} The operators $Re(A),Im(A),\text{ and }I$ are contained in $V_A$.\end{lem}
Proof: Considering Corollary \ref{decomp}, we see
\begin{align*}
    \phi_{Re(A)}&=\pi_x\circ \phi_A\\
    \phi_{Im(A)}&=-i\pi_y\circ \phi_A\\
    \phi_{I}&=1_\mathbb{C}\circ \phi_A,
    \end{align*}
    where $\pi_x$ denotes projection onto the $x$-coordinate and $1_{\mathbb{C}}$ is the constant map $1$. $\square$
        \section{The $M_2(\mathbb{C})$ Case}\label{2dcase}
    Similar to a popular proof for the Toeplitz-Hausdorff theorem, we calculate all of the details for 2x2 matrices, and our main result follows from compressing to a two dimensional subspace. 
    \begin{prop}\label{dimension} For $A\in M_2(\mathbb{C})$, the dimension of $V_A$ is either 1,2 or 3. \end{prop}
    The dimensions correspond to three possibilities. If $A=\lambda I_2$, then $V_A$ is dimension 1. If $A$ is normal, but not a scalar multiple of identity, then $V_A$ is dimension 2. And if $A$ is not normal, $V_A$ is dimension 3. \\\\
    We state two essential lemmas and explore these vector spaces before proving Proposition \ref{dimension}.
    \begin{lem}\label{ellipse} If $A\in M_2(\mathbb{C})$ is not normal, then $W(A)$ is a non-degenerate ellipse.\end{lem} 
   Prro: See \cite{yuangao}.
    \begin{lem}\label{tsing} Let $A\in M_2(\mathbb{C})$ be non-normal. Suppose $z\in W(A)$. Then \\
    i) If $z\in\partial W(A)$ then there is exactly one (up to scalar multiples) $x\in S_2$ such that $\phi_A(x)=z$\\
    ii) If $z\notin\partial W(A)$ then there are exactly two (up to scalar multiplies) $x\in S_2$ such that $\phi_A(x)=z$.\end{lem}
    Proof: See \cite{tsing}.
    \begin{lem}\label{constant}If $A=\lambda I_2$ and $B\leq A$, then $B=\tau I_2$.\end{lem}
    Proof: From Proposition \ref{basics}, $\phi_A$ is the constant map. Hence, if $B\leq A$, then $\phi_B$ must be a constant map, and $B=\tau I_2$. $\square$ \\\\
       We explore the numerical range map and associated vector space for normal matrices that are not scalar multiples of the identity. It is well known that the numerical range of a normal matrix is the convex hull of its eigenvalues \cite{yuangao}. For a normal 2x2 matrix $A$, we have $W(A)=[\lambda_1,\lambda_2]$, where $\lambda_i$ are the eigenvalues of $A$, and $[\lambda_1,\lambda_2]$ denotes the closed line segment between $\lambda_1$ and $\lambda_2$.   \\\\
    The following lemma mirrors Lemma \ref{tsing}, but was not proven in \cite{tsing}. 
    \begin{lem}\label{norm1} For a normal, not scalar of identity, $A\in M_2(\mathbb{C})$ with eigenvalues $\lambda_1,\lambda_2$\\
    i) There is exactly one $x\in S_2$, up to scaling, such that $\phi_A(x)=\lambda_i$, and this $x$ is necessarily an eigenvector of $A$.\\
    ii) For $z\in (\lambda_1,\lambda_2)$, there are uncountably many $x\in S_2$, up to scaling, such that $\phi_A(x)=z$. \end{lem}
        Proof: We first consider the diagonal matrix. Let 
    \begin{equation*}
        A=\begin{bmatrix}
            \lambda_1 & 0 \\
            0 & \lambda_2
        \end{bmatrix}.
    \end{equation*}
    We see that 
    \begin{align*}
        \left\langle A \begin{bmatrix}
            x_1 \\ x_2
        \end{bmatrix}, \begin{bmatrix}
            x_1 \\ x_2
        \end{bmatrix}\right\rangle &=\left\langle \begin{bmatrix}
            \lambda_1 x_1 \\ \lambda_2 x_2\end{bmatrix},\begin{bmatrix}
                x_1 \\ x_2
            \end{bmatrix}\right\rangle \\
               &= \lambda_1|x_1|^2+\lambda_2|x_2|^2.
    \end{align*}
       From the above we have $\phi_A(x)=\lambda_i$ if and only if $|x_i|^2=1$, hence if and only if $x$ is an eigenvector of $A$. \\\\
    Now consider $z\in (\lambda_1,\lambda_2)$. There exists a $t\in(0,1)$ such that $z=t\lambda_1+(1-t)\lambda_2$. We see that for all $\theta\in[0,2\pi)$
    \begin{equation*}
        \phi_A\left(\begin{bmatrix}
            e^{i\theta}\sqrt{t}\\ \sqrt{1-t}
        \end{bmatrix}\right)=z.
    \end{equation*}
    Up to scaling, $\phi_A$ maps one vector to each end point and an uncountable number to interior points. \\\\
    As all normal matrices are unitarily equivalent to a diagonal matrix, the proof for a general normal matrix follows from the above. $\square$
    \begin{cor}\label{norm2}If $A$ is normal and $B\leq A$, then $B$ is normal.\end{cor}
    Proof: From Lemma \ref{tsing}, we know that the preimage of any value in the numerical range of a non-normal 2x2 matrix has either one or two elements up to scaling. From Lemma \ref{norm1}, normal matrices pair uncountably many elements up to scaling. $\square$
\begin{cor}\label{norm3} If $A$ is normal and $B\leq A$, then $B$ has the same eigenvectors as $A$. \end{cor}
Proof: If $B=\lambda I_2$, then the claim is true. \\\\
If $B$ is not a scalar multiple of the identity, from Corollary \ref{norm2}, $B$ is normal. From Lemma \ref{norm1} $B$'s eigenvectors are not paired with any other vectors. Hence, from Lemma \ref{norm1} they must be eigenvectors of $A$. $\square$.
 \begin{lem}\label{nonnorm}If $A\in M_2(\mathbb{C})$ is not normal, there exists a $B$ such that $B\not\leq A$. \end{lem}
 Proof: From Lemma \ref{ellipse} the numerical range of $A$ is a nondegenerate ellipse. Take a point $u$ from the interior of $W(A)$ and $v$ from the boundary. From Lemma \ref{tsing}, there exist two linearly independent unit vectors $x,y$ such that $\phi_A(x)=\phi_A(y)=u$ and a unit vector $z$ such that $\phi_A(z)=v$. Consider a unitary matrix $U$ such that $Ux=z$. \\\\
 From Proposition \ref{basics}, we know $W(U^*AU)=W(A)$. From our construction, we see that 
 \begin{align*}
     \phi_{U^*AU}(x)&:=\langle U^*AUx,x\rangle\\
     &=\langle AUx,Ux\rangle\\
     &=\langle Az,z\rangle\\
     &=v,
 \end{align*}
 a boundary point of $W(U^*AU)$. From Lemma \ref{tsing}, we know that $\phi_{U^*AU}$ does not pair $x$ with another vector, however, $\phi_A$ pairs $x$ with $y$. Hence, $U^*AU\not\leq A$. $\square$
 \begin{lem}\label{independent} If $A\in M_2(\mathbb{C})$ is not normal, then $Re(A),Im(A),$ and $I$ are linearly independent.\end{lem}
 Proof: If $Re(A)$ were a linear combination of $Im(A)$ and $I_2$, then $Re(A)$ would commute with $Im(A)$, and $A$ would be normal. The same proof works for $Im(A)$. \\\\
 If $I_2$ were a linear combination of $Re(A)$ and $Im(A)$, then $Re(A)$ and $Im(A)$ must be simultaneously diagonalizable and would commute. $\square$\\\\
     \textbf{Proof of Proposition \ref{dimension}}\\\\
          A matrix $A\in M_2(\mathbb{C})$ falls into one of these three cases. \\\\
     Case 1. If $A=\lambda I_2$, by Lemma \ref{constant}, we know that $B\leq A$ implies $B=\tau I_2$. Hence, $V_A$ is generated by $I_2$. \\\\
     Case 2. If $A$ is normal, but not in case 1, then the linearly independent elements $A$ and $I_2$ are in $V_A$. By Corollary \ref{norm2} and Corollary \ref{norm3}, we know that every element of $V_A$ is a normal matrix with the same eigenvectors as $A$. We see that, if $A$ is diagonal, then $B\in V_A$ is also diagonal. This is a two dimensional complex vector space, hence must be spanned by $A$ and $I_2$.  \\\\
         Case 3. If $A$ is not normal, $V_A$ contains $Re(A)$, $Im(A)$, and $I_2$. From Lemma \ref{independent} these are linearly independent. So, $V_A$ is of degree at least three. By Lemma \ref{nonnorm}, we know that there is a matrix not in $V_A$, so it is not of degree four. Hence, $V_A$ is generated by $Re(A)$, $Im(A)$, and $I_2$. $\square$
    \begin{cor}\label{affine1} The connecting maps for 2x2 matrices are affine.\end{cor}
    Proof. By Proposition \ref{dimension} if $B\leq A$, then 
    \begin{equation*}
        B=(a_1+ia_2)Re(A)+(b_1+ib_2) Im(A)+(c_1 +ic_2)I_2
    \end{equation*}
for real numbers $a_1,a_2,b_1,b_2,c_1,c_2$. Considering Corollary \ref{decomp}, the connecting map between $A$ and $B$ is of the form
\begin{equation*}
    f\left(\begin{bmatrix}
        x_1\\x_2
    \end{bmatrix}\right)=\begin{bmatrix}
        a_1 & b_1 \\ a_2 & b_2
    \end{bmatrix}\begin{bmatrix}
        x_1\\x_2
    \end{bmatrix}+\begin{bmatrix}
        c_1\\c_2
    \end{bmatrix}.
\end{equation*}
See that for $x\in S_2$
\begin{align*}
    f\circ\phi_A(x)&=f\left(\begin{bmatrix}
        \phi_{Re(A)}(x)\\\phi_{Im(A)}(x)
    \end{bmatrix}\right)\\
    &=(a_1+ia_2)\phi_{Re(A)}(x)+(b_1+ib_2) \phi_{Im(A)}(x)+(c_1 +ic_2)\\
    &=\phi_B(x)
\end{align*}
$\square$
    \section{Connecting Maps are Affine}\label{conn}
    In this section, we prove one of our main results: if $\phi_B=f\circ \phi_A$, then $f$ is affine. We will see that by compressing the problem to two dimensions, this result follows easily from our work in section \ref{2dcase}. 
    \begin{lem}\label{comp2}For a closed subspace $M\subset H$, if $B\leq A$, then $B_M\leq A_M$. \end{lem}
    Proof: The connecting map between $A$ and $B$, restricted to $W(A_M)$, connects $A_M$ and $B_M$. $\square$
    \begin{thm}\label{affine2}If $A,B\in B(H)$ and $\phi_B=f\circ \phi_A$, then $f$ is affine. \end{thm}
       Proof: Take $a,b\in W(A)$ and $t\in (0,1)$. We have $x,y\in S_H$, necessarily independent, such that $\phi_A(x)=a$ and $\phi_A(y)=b$. Let $M$ be the two-dimensional subspace spanned by $x$ and $y$. By the convexity of the numerical range, $ta+(1-t)b\in W(A_M)$, and by our work in the two-dimensional case, the map sending $\langle A_Mx,x\rangle\to \langle B_Mx,x\rangle$ is affine. Hence, the connecting map from $A$ to $B$ preserves affine combinations. $\square$  
            \begin{cor}\label{lincomb} If $B\leq A$, then 
    \begin{equation*}
        B=\alpha Re(A)+\beta Im(A)+\gamma I
    \end{equation*}
    for some scalars $\alpha,\beta,\gamma\in\mathbb{C}$. \end{cor}
    Proof: As the connecting map from $A$ to $B$ is affine, there exists a real valued matrix and vector such that 
\begin{equation*}
    f\left(\begin{bmatrix}
        x_1\\ x_2
    \end{bmatrix}\right)=\begin{bmatrix}
        a_1 & b_1 \\ a_2 & b_2
    \end{bmatrix}\begin{bmatrix}
        x_1\\ x_2
    \end{bmatrix}+\begin{bmatrix}
        c_1 \\ c_2
    \end{bmatrix}.
\end{equation*}
So $\forall x\in S_H$ we have 
\begin{align*}
    \langle Bx,x\rangle&= f\left(\langle Ax,x\rangle\right)\\
    &=f\left(\begin{bmatrix}
        \langle Re(A)x,x\rangle \\ \langle Im(A)x,x\rangle
    \end{bmatrix}\right)\\
    &=\begin{bmatrix}
        a_1\langle Re(A)x,x\rangle +b_1 \langle Im(A)x,x\rangle +c_1\\
        a_2\langle Re(A)x,x\rangle + b_2 \langle Im(A)x,x\rangle +c_2
    \end{bmatrix}\\
    &= a_1\langle Re(A)x,x\rangle +b_1 \langle Im(A)x,x\rangle +c_1+i(a_2\langle Re(A)x,x\rangle + b_2 \langle Im(A)x,x\rangle +c_2)\\
    &=\langle ((a_1+ia_2)Re(A)+(b_1+ib_2)Im(A)+(c_1+ic_2)I)x,x\rangle.
\end{align*}
As the numerical range maps are identical, we have 
\begin{equation*}
    B=((a_1+ia_2)Re(A)+(b_1+ib_2)Im(A)+(c_1+ic_2)I).\square
\end{equation*} 
\begin{cor}If $B\leq A$ and $\phi_A(x)$ is an extreme point of $W(A)$, then $\phi_B(x)$ is an extreme point of $W(B)$.\end{cor}
Proof: This is a property of affine transformations. $\square$

\begin{cor}If $B\leq A$ and $A\leq B$, then $W(A)$ is homeomorphic to $W(B)$. \end{cor}
Proof: Since the map connecting $W(A)$ to $W(B)$ is an invertible affine map, it is a homeomorphism between $W(A)$ and $W(B)$. $\square$
\section{Numerical Information of an Operator}\label{numinfo}
\begin{defi}\label{numinf} For $A\in B(H)$ the Numerical Information of $A$, denoted $\Sigma_A$, is the set
\begin{equation*}
    \Sigma_A:=\{\phi_A^{-1}(U)\;|\;U\subset W(A)\text{ is measurable}\}.
\end{equation*}\end{defi}
We need the following important lemma to prove the remaining results. \\\\
\textbf{Doob-Dynkin Lemma} $\Sigma_B\subset \Sigma_A$ if and only if there exists a measurable function $f$ such that $\phi_B=f\circ \phi_A$. \\
Proof: See \cite{modprob}.
\begin{thm}\label{inclusion}For $A,B\in B(H)$, $\Sigma_B\subset \Sigma_A$ if and only if \begin{equation*}
    B=\alpha Re(A)+\beta Im(A)+\gamma I
\end{equation*}.\end{thm}
Proof: Assume that $B$ is of this form, from Corollary \ref{decomp}, we can construct our connecting map. By Doob-Dynkin, $\Sigma_B\subset\Sigma_A$. \\\\
Now assume that $\Sigma_B\subset\Sigma_A$. From Doob-Dynkin, the connecting map is well defined. From Corollary \ref{lincomb} we get our result. $\square$
\begin{thm}For $A,B\in 
B(H)$, $\Sigma_A=\Sigma_B$ if and only if \begin{equation*}
    B=(a_1+ia_2)Re(A)+(b_1+ib_2)Im(A)+(c_1+ic_2) I\text{ and }a_1b_2-a_2b_1\neq 0
\end{equation*}.\end{thm}
Proof: If we assume $B=(a_1+ia_2)Re(A)+(b_1+ib_2)Im(A)+(c_1+ic_2) I$ and $a_1b_2-a_2b_1\neq 0$, by Theorem \ref{inclusion}, $\Sigma_B\subset \Sigma_A$. The map connecting $B$ and $A$ is of the form
\begin{equation*}
    f\left(\begin{bmatrix}
        x_1\\x_1
    \end{bmatrix}\right)=\begin{bmatrix}
        a_1 & b_1 \\a_2 & b_2
    \end{bmatrix}\begin{bmatrix}
        x_1\\x_2
    \end{bmatrix}+\begin{bmatrix}
        c_1\\c_2
    \end{bmatrix}.
\end{equation*}
As $a_1b_2-a_2b_1\neq 0$, this map is invertible. Hence $\phi_A=f^{-1}\circ \phi_B$ and $\Sigma_A\subset \Sigma_B$. \\\\
Now, if we assume $\Sigma_A=\Sigma_B$, then there exists $f$ such that $\phi_B=f\circ \phi_A$ and $g$ such that $\phi_A=g\circ\phi_B$. It is necessary that $f\circ g=Id$ on $W(A)$ and $g\circ f=Id$ on $W(B)$ as 
\begin{equation*}
    \phi_B=f\circ\phi_A=f\circ g\circ \phi_B\text{ and }\phi_A=g\circ\phi_B=g\circ f\circ \phi_A.
\end{equation*}
As $f$ is invertible, the result follows easily from Theorem \ref{inclusion}. $\square$
\\\\
We extend the Numerical Information to the Joint Numerical Information. \\\\
For two operators $A,B\in B(H)$, we use the notation $\phi_{A,B}$ to denote the joint numerical range map of $A$ and $B$:
\begin{align*}
    \phi_{A,B}:S_H&\to \mathbb{C}\times \mathbb{C}\\
    x&\mapsto \phi_A(x)\times\phi_B(x).
\end{align*}
\begin{defi}For $A,B\in B(H)$ we define the joint numerical information of $A$ and $B$ as
\begin{equation*}
    \Sigma_{A,B}:=\{\phi_{A,B}^{-1}(U)\;|\;U\in W(A,B)\text{ measurable}\}.
\end{equation*}\end{defi}
\begin{lem}The Joint Numerical Range $\Sigma_{A,B}$ is the smallest sigma algebra that contains $\Sigma_A\cup \Sigma_B$. \end{lem}
Proof: We again apply Doob-Dynkin. As
\begin{equation*}
    \phi_A=\pi_1\circ\phi_{A,B}.
\end{equation*}
where $\pi_1$ denotes the projection onto the first coordinate, we have $\Sigma_A\subset \Sigma_{A,B}$. We can mimic the proof above for the containment of $\Sigma_B$. \\\\
If we pull back the product of measurable sets $X\subset W(A)$ and $Y\subset W(B)$ by the joint numerical range map we have
\begin{equation*}
    \phi_{A,B}^{-1}(X\times Y)=\phi_A^{-1}(X)\cap \phi_B^{-1}(Y),
\end{equation*}
which is the intersection of an element from $\Sigma_A$ with one from $\Sigma_B$. As products of measurable sets from $W(A)$ and $W(B)$, (necessarily intersected with $W(A,B)$), generate the sigma algebra on $W(A,B)$, their pullbacks generate $\Sigma_{A,B}$. As the generators of $\Sigma_{A,B}$ can be written as the intersection of an element of $\Sigma_A$ and $\Sigma_B$, $\Sigma_{A,B}$ is contained in the sigma algebra generated by $\Sigma_A\cup \Sigma_B$. $\square$
\begin{cor} For $A\in B(H)$, $\Sigma_A$ is the smallest sigma algebra containing $\Sigma_{Re(A)}$ and $\Sigma_{Im(A)}$.\end{cor}
 
\section{Acknowledgments}
The author gratefully acknowledges support from the Erwin and Peggy Kleinfeld Graduate Fellowship fund, support from a GAANN fellowship (U.S. Department of Education), and Dr. Palle. E.T Jorgensen for his guidance.

\end{document}